\documentclass[11pt]{article}
\usepackage{amsmath}
\usepackage{amssymb}
\usepackage{amsthm}
\newtheorem{theorem}{Theorem}[section]
\newtheorem{lemma}[theorem]{Lemma}
\newtheorem{corollary}[theorem]{Corollary}
\newtheorem{proposition}[theorem]{Proposition}
\newtheorem{example}[theorem]{Example}
\newtheorem{problem}[theorem]{Problem}

\newtheorem*{thmsc}{Theorem SC}
\newtheorem*{thmsas}{Theorem S}
\newtheorem*{thmms}{Theorem MS}
\newtheorem*{thmdc}{Theorem DC}
\newtheorem*{thmaag}{Theorem A\!$^2$G}

\numberwithin{equation}{section}

\long\def\symbolfootnote[#1]#2{\begingroup%
\def\thefootnote{\fnsymbol{footnote}}\footnote[#1]{#2}\endgroup}

\begin{document}

\def\C{{\mathbb C}}
\def\N{{\mathbb N}}
\def\Z{{\mathbb Z}}
\def\R{{\mathbb R}}
\def\B{{\cal B}}
\def\P{{\cal P}}
\def\epsilon{\varepsilon}
\def\phi{\varphi}
\def\leq{\leqslant}
\def\geq{\geqslant}
\def\slim{\mathop{\hbox{$\overline{\hbox{\rm lim}}$}}\limits}
\def\ilim{\mathop{\hbox{$\underline{\hbox{\rm lim}}$}}\limits}
\def\dim{\hbox{\tt dim}\,}
\def\spann{\hbox{\tt span}\,}

\title{A weighted bilateral shift with cyclic square is supercyclic}

\author{Stanislav Shkarin}

\date{}

\maketitle

\begin{abstract} It is shown that for a bounded weighted bilateral
shift $T$ acting on $\ell_p(\Z)$ for $1\leq p\leq 2$ supercyclicity
of $T$, weak supercyclicity of $T$, cyclicity of $T\oplus T$ and
cyclicity of $T^2$ are equivalent. A new sufficient condition for
cyclicity of a weighted bilateral shift is proved, which implies, in
particular, that any compact weighted bilateral shift is cyclic.
\end{abstract}

\small \noindent{\bf MSC:} \ \ 47A16, 37A25

\noindent{\bf Keywords:} \ \ Cyclicity, Supercyclicity,
Hypercyclicity, Quasisimilarity, Weighted bilateral shifts
\normalsize

\section{Introduction \label{s1}}\rm

Throughout\symbolfootnote[0]{Partially supported by Plan Nacional
I+D+I grant no. MTM2006-09060 and Junta de Andaluc\'{\i}a FQM-260
and British Engineering and Physical Research Council Grant
GR/T25552/01.} this article all vector spaces are supposed to be
over the field $\C$ of complex numbers, $\Z$ is the set of integers,
$\Z_+$ is the set of non-negative integers and $\N$ is the set of
positive integers. As usual, symbol $L(\B)$ stands for the space of
bounded linear operators on a Banach space $\B$ and $\B^*$ is the
space of continuous linear functionals on $\B$.

For $w\in\ell_\infty(\Z)$ and $1\leq p\leq\infty$, $T_{w,p}$ stands
for the bounded linear operator acting on $\ell_p(\Z)$ if $1\leq
p<\infty$ or $c_0(\Z)$ if $p=\infty$, defined on the canonical basis
$\{e_n\}_{n\in\Z}$ by
$$
T_{w,p}e_n=w_ne_{n-1}\qquad \text{for\ \ }n\in\Z.
$$
If additionally $w_n\neq 0$ for each $n\in\Z$, the operator
$T_{w,p}$ is called the {\it weighted bilateral shift} with the
weight sequence $w$. In the particular case $w_n\equiv1$ we have the
{\it unweighted bilateral shift}.

Recall that a bounded linear operator $T$ on a Banach space $\B$ is
called {\it cyclic} if there exists $x\in \B$ such that
$\spann\{T^nx:n\in\Z_+\}$ is dense in $\B$. $T$ is called {\it
supercyclic} if there is $x\in \B$ for which $\{\lambda
T^nx:\lambda\in\C,\ n\in\Z_+\}$ is dense in $\B$. Similarly $T$ is
called {\it hypercyclic} if there is $x\in \B$ such that the orbit
$\{T^nx:x\in\Z_+\}$ is dense in $\B$. Finally $T$ is called {\it
weakly supercyclic} or {\it weakly hypercyclic} if the density is
required with respect to the weak topology. We refer to surveys
\cite{ge1,ge2,msa} for additional information on hypercyclicity and
supercyclicity. One of the attractive features of weakly supercyclic
operators is that all their powers are again weakly supercyclic and
therefore cyclic. For norm supercyclicity this statement was proved
by Ansari \cite{ansa} and the same proof works for weak
supercyclicity.

\smallskip

Cyclicity properties of weighted bilateral shifts have been
intensely studied. Hypercyclicity and supercyclicity of weighted
bilateral shifts were characterized by Salas \cite{sal,sal1} in
terms of the weight sequences. It was observed in
\cite[Proposition~5.1]{shk1} that the Salas conditions admit the
following simpler equivalent form.

\begin{thmsas}
For $1\leq p\leq \infty$, a weighted bilateral shift $T=T_{w,p}$ is
hypercyclic if and only if for any $m\in\Z_+$,
\begin{equation}
\ilim\limits_{n\to\infty} \max\{\widetilde{w}(m-n+1,m),
(\widetilde{w}(m+1,m+n))^{-1}\}=0 \label{sal3}
\end{equation}
and $T$ is supercyclic if and only if for any $m\in\Z_+$,
\begin{equation}
\ilim\limits_{n\to+\infty}
\widetilde{w}(m-n+1,m)\widetilde{w}(m+1,m+n)^{-1}=0, \label{sal4}
\end{equation}
where
\begin{equation}
\widetilde{w}(a,b)=\prod_{j=a}^b\,|w_j|\quad \ \ \text{for}\ \
a,b\in\Z\ \text { with } \ a\leq b. \label{beta}
\end{equation}
\end{thmsas}

On the other hand, cyclicity of a weighted bilateral shift turns out
to be a much more subtle issue, see, for instance,
\cite{her0,her2,nik,shields}. It is worth noting that unlike for
hypercyclicity or supercyclicity, cyclicity of a weighted bilateral
shift depends on $p$. For instance, the unweighted bilateral shift
is cyclic on $\ell_2(\Z)$ and non-cyclic on $\ell_1(\Z)$. There
exist several necessary and several sufficient conditions of
cyclicity of a weighted bilateral shift, see for instance, the works
of Herrero \cite{her0,her2}. Prominent among these conditions is the
observation that if the adjoint of a weighted bilateral shift $T$ on
$\ell_p(\Z)$ for $1\leq p<\infty$ has non-empty point spectrum then
$T$ is non-cyclic. This implies, in particular, that the weighted
bilateral shift $T_{w,p}$ with the weight sequence $w_n=a$ for
$n\leq 0$ and $w_n=b$ for $n>0$ with $0<|b|<|a|$ is non-cyclic for
any $p\in[1,\infty]$. This is precisely the shape of the first
example of a non-cyclic weighted bilateral shift, obtained by
Beauzamy \cite{boze}.

Recall that a bounded linear operator $T$ on a Banach space $\B$ is
said to satisfy the {\it Supercyclicity Criterion} \cite{msa} if
there exist a strictly increasing sequence $\{n_k\}_{k\in\Z_+}$ of
positive integers, dense subsets $E$ and $F$ of $\B$ and a map
$S:F\to F$ such that $TSy=y$ for each $y\in F$ and
$\|T^{n_k}x\|\|S^{n_k}y\|\to 0$ as $k\to \infty$ for any $x\in E$
and $y\in F$. The following two results are proved in \cite{msa}.

\begin{thmsc}\it Any operator satisfying the Supercyclicity
Criterion is supercyclic. \end{thmsc}\rm

\begin{thmms}A weighted bilateral shift on $\ell_p(\Z)$ for $1\leq p<\infty$ or
on $c_0(\Z)$ is supercyclic if and only if it satisfies the
Supercyclicity Criterion. \end{thmms}\rm

There is no great mystery about the last theorem. One just has to
take $E=F$ being the space of sequences with finite support, $S$
being the inverse of the restriction of $T$ to $F$ and use Theorem~S
to find an appropriate sequence $\{n_k\}$. Note also that weak
hypercyclicity and weak supercyclicity of weighted bilateral shifts
were studied in \cite{bm,cs,san1,shk1}. In \cite{shk1} it is shown
that for $p\leq 2$ any weighted bilateral shift on $\ell_p(\Z)$ is
either supercyclic or is not weakly supercyclic. We extend this
dichotomy.

\begin{theorem}\label{dich}Let $1\leq p\leq 2$ and $T$ be a weighted bilateral
shift on $\ell_p(\Z)$. Then the following statements are equivalent:

\smallskip

\noindent{\rm(C1)} \ $T$ satisfies the Supercyclicity Criterion$;$

\smallskip

\noindent{\rm(C2)} \ $T$ is supercyclic$;$

\smallskip

\noindent{\rm(C3)} \ $T$ is weakly supercyclic$;$

\smallskip

\noindent{\rm(C4)} \ $T\oplus T$ is cyclic$;$

\smallskip

\noindent{\rm(C5)} \ there is $n\geq 2$ for which $T^n$ is cyclic$;$

\smallskip

\noindent{\rm(C6)} \ for any $n\in\N$, $T^n$ is cyclic.
\end{theorem}

We would like to stress that, as it is shown in \cite{shk1}, for
each $p>2$ there is a weakly supercyclic non-supercyclic weighted
bilateral shift $T$ on $\ell_p(\Z)$. Since powers of a weakly
supercyclic operators are cyclic, we see that (C6) does not imply
(C2) when $p>2$. From this observation and equivalence of (C5), (C6)
and (C2) for $p\leq 2$, we immediately obtain the following
corollary.

\begin{corollary}\label{square}A weighted bilateral shift $T$ acting
on $\ell_p(\Z)$ with $1\leq p\leq 2$ is supercyclic if and only if
$T^2$ is cyclic. On the other hand for each $p>2$ there exists a
non-supercyclic weighted bilateral shift on $\ell_p(\Z)$ whose
powers are all cyclic.
\end{corollary}

It is also worth noting that a sufficient condition of weak
supercyclicity of weighted bilateral shifts $T$ in \cite{shk1}
automatically gives weak supercyclicity and therefore cyclicity of
$T\oplus T$. On the other hand, weakly supercyclic non-supercyclic
operators $T$ constructed in \cite{san1,bm} have the property that
$T\oplus T$ is not cyclic. In {\rm \cite{unicell}} a sufficient
condition for a weighted bilateral shift to be unicellular (and
therefore cyclic) is given. This result together with Theorem~S
imply that there are cyclic non-supercyclic weighted bilateral
shifts on $\ell_p(\Z)$ for $1\leq p<\infty$ and on $c_0(\Z)$. Thus
the condition $n\geq 2$ in (C5) is essential. From the proof of
Theorem~\ref{dich} below it will be clear which relations between
the conditions (C1--C6) hold for any bounded linear operator on a
separable Banach space. We shall also show that the implication
${\rm (C5)}\,\Longrightarrow\,{\rm (C4)}$ is satisfied for any
weighted bilateral shift $T=T_{w,p}$ with $1\leq p\leq \infty$. The
last implication though is not satisfied for general operators. For
instance, the Volterra operator $Vf(t)=\int_0^x f(t)\,dt$, acting on
$L_2[0,1]$, satisfies (C6) and does not satisfy (C4), see, for
instance, \cite{mos}.

Finally, we shall prove yet another sufficient condition for
cyclicity of a weighted bilateral shift. It does not follow from any
known sufficient condition including the following most recent one
due to Abakumov, Atzmon and Grivaux \cite{aba}.

\begin{thmaag} \it Let $w=\{w_n\}_{n\in\Z}$ be a bounded
sequence of non-zero complex numbers, $\alpha_0=1$,
$\alpha_n=(\widetilde w(1,n))^{-1}$ for $n>0$ and
$\alpha_n=\widetilde w(1+n,0)$ for $n<0$, where the numbers
$\widetilde w(a,b)$ are defined in $(\ref{beta})$. Assume also that
there exist $k\in\N$ and a submultiplicative sequence
$\{\rho_n\}_{n\in\Z_+}$ of positive numbers such that
$\ln^+(\rho_n)=o(\sqrt{n})$, $\alpha_{-n}=O(n^k)$ and
$\alpha_n=O(\rho_n)$ as $n\to+\infty$. Then the weighted bilateral
shift $T=T_{w,p}$ is cyclic if and only if the sequence
$\{\alpha_n^{-1}\}_{n\in\Z}$ does not belong to $\ell_q(\Z)$, where
$\frac1p+\frac1q=1$.
\end{thmaag} \rm

This highly non-trivial result does not give a characterization of
cyclicity for weighted bilateral shifts. For instance, the
conditions imposed on the weight sequence rule out compact weighted
bilateral shifts. The next theorem can be applied for a wider
variety of weight sequences, although becomes a weaker statement,
when applied to the weights satisfying the conditions of
Theorem~A\!$^2$\nobreak\hskip-1pt\nobreak G.

\begin{theorem} \label{suco} Let $w=\{w_n\}_{n\in\Z}$ be a bounded
sequence of non-zero complex numbers such that for any $a\in\N$,
\begin{equation}\label{a123}
\inf\{\widetilde w(1,m)^{-1}\widetilde w(-j(m-a),0)^{1/j}:j\in\N,\
m\geq a\}=0.
\end{equation}
Then the weighted bilateral shift $T_{w,p}$ is cyclic for $1\leq
p\leq \infty$.
\end{theorem}

Substituting $m=a+1$ into (\ref{a123}), we immediately obtain the
following corollary.

\begin{corollary} \label{suco1} Let $w=\{w_n\}_{n\in\Z}$ be a bounded
sequence of non-zero complex numbers such that
\begin{equation}\label{b123}
\ilim_{n\to+\infty}\widetilde w(1-n,0)^{1/n}=0.
\end{equation}
Then the weighted bilateral shift $T_{w,p}$ is cyclic for $1\leq
p\leq\infty$.
\end{corollary}

Since $\|T_{w,p}^n\|\geq \|T_{w,p}^ne_0\|=\widetilde w(1-n,0)$ for
each $n\in\Z_+$, the spectral radius formula implies that any
quasinilpotent weighted bilateral shift satisfies (\ref{b123}). Note
also that any compact weighted bilateral shift is quasinilpotent.
Thus, the following corollary holds true.

\begin{corollary} \label{sucoco} Any quasinilpotent weighted bilateral
shift is cyclic. In particular, any compact weighted bilateral shift
is cyclic. \end{corollary}

If we fix $j\in\N$ in (\ref{a123}), we immediately obtain the
following corollary.

\begin{corollary} \label{suco2} Let $w=\{w_n\}_{n\in\Z}$ be a bounded
sequence of non-zero complex numbers for which there exists $j\in\N$
such that
\begin{equation}\label{c123}
\ilim_{m\to+\infty}\frac{\widetilde w(a-jm,0)}{(\widetilde
w(1,m))^j}=0\quad\text{for each}\quad a\in\N.
\end{equation}
Then the weighted bilateral shift $T_{w,p}$ is cyclic for $1\leq
p\leq\infty$.
\end{corollary}

\begin{example} \rm Let $a,b>0$, $0<\alpha\leq 1$ and
$w=\{w_n\}_{n\in\Z}$ be a sequence of positive numbers such that
$1-w_n\sim an^{-\alpha}$ as $n\to+\infty$ and $1-w_n\sim
b(-n)^{-\alpha}$ as $n\to -\infty$. From Corollary~\ref{suco2} it
follows easily that the weighted bilateral shift $T_{w,p}$ is cyclic
for $1\leq p\leq \infty$. On the other hand,
Theorem~A\!$^2$\nobreak\hskip-1pt\nobreak G is applicable only if
$\alpha>1/2$. Note also that by Theorem~S, $T$ is supercyclic if
$b>a$ and is non-supercyclic if $b<a$.
\end{example}

\section{Proof of Theorem~\ref{dich} \label{s2}}

We start with the following three easy and known but nice
observations.

\begin{lemma}\label{quasi} Let $\B_1$ and $\B_2$ be Banach spaces
and $T_1\in L(\B_1)$, $T_2\in L(\B_2)$ be such that there exists a
bounded linear operator $J:\B_1\to \B_2$ with dense range satisfying
$T_2J=JT_1$. Then cyclicity of $T_1$ implies cyclicity of $T_2$.
\end{lemma}

\begin{proof} Observe that
$\spann\{T^n_2 Jx:n\in\Z_+\}=J\bigl(\spann\{T_1^nx:n\in\Z_+\}\bigr)$
for each $x\in\B_1$. Hence $Jx$ is a cyclic vector for $T_2$ for
each cyclic vector $x$ for $T_1$. \end{proof}

\noindent{\bf Remark.} \ The same argument shows that
Lemma~\ref{quasi} remains true if cyclicity is replaced by
hypercyclicity, supercyclicity, weak hypercyclicity or weak
supercyclicity.

\begin{lemma}\label{adj} Let $\B$ be a Banach space and $T\in
L(\B)$. Then the operator $T\oplus T^*$ acting on $\B\times \B^*$ is
non-cyclic. \end{lemma}

\begin{proof}Let $(x,f)\in \B\times \B^*$ be different from zero.
Then the continuous linear functional $F$ on $\B\oplus\B^*$ defined
by $F(y,g)=f(y)-g(x)$ is non-zero. We have,
\begin{equation*}
F(T^nx,T^{*n}f)=f(T^nx)-T^{*n}f(x)=f(T^nx)-f(T^nx)=0\quad\text{for
any}\quad n\in\Z_+.
\end{equation*}
Thus, the  orbit of any non-zero vector  under  $T\oplus T^* $ is
contained in the kernel of a non-zero continuous linear functional.
Therefore, $T\oplus T^*$ is not cyclic. \end{proof}

\begin{corollary}\label{tplust} Let $\B$ be a Banach space and
$T\in L(\B)$ be such that there exists a bounded linear operator
$J:\B\to\B^*$ with dense range satisfying $T^*J=JT$. Then the
operator $T\oplus T$ acting on $\B\oplus \B$ is non-cyclic.
\end{corollary}

\begin{proof} Since $T^*J=JT$, we have $(T\oplus T^*)(I\oplus J)=
(I\oplus J)(T\oplus T)$. Assume that $T\oplus T$ is cyclic. Since
$I\oplus J:\B\times \B\to \B\times \B^*$ is bounded and has dense
range, Lemma~\ref{quasi} implies that $T\oplus T^*$ is cyclic, which
is impossible according to Lemma~\ref{adj}. \end{proof}

\begin{lemma} \label{pow} Let $j\in\N$ and $T$ be a bounded linear
operator with dense range on a Banach space $\B$ such that $T^j$ is
cyclic. Let also $z=e^{2\pi i/j}$. Then the operator
$$
S=T\oplus zT\oplus z^2 T\oplus {\dots} \oplus z^{j-1}T,
$$
acting on $\B^j$, is cyclic. \end{lemma}

\begin{proof} Let $x$ be a cyclic vector for $T^j$. Then
$L=\{r(T^j)x:r\in\P\}$ is dense in $\B$, where $\P=\C[t]$ is the
space of polynomials on one variable with complex coefficients.
Since $T$ has dense range, the spaces $T(L),\dots,T^{j-1}(L)$ are
also dense in $\B$. It suffices to verify that $u=(x,x,\dots,x)\in
\B^j$ is a cyclic vector for $S$. Let $M$ be the closed linear span
of the orbit of $u$ under $S$, $0\leq k\leq j-1$ and $r\in \P$. Then
$$
S^kr(S^j)u=(T^kr(T^j)x,z^kT^kr(T^j)x,\dots,z^{k(j-1)}T^kr(T^j)x)\in
M.
$$
Thus, $M$ contains the vectors of the shape
$(a,z^ka,\dots,z^{k(j-1)}a)$ for $a\in T^k(L)$ and $0\leq k\leq
j-1$. Since $M$ is closed and $T^k(L)$ is dense in $\B$, we see that
$$
M\supseteq N_k=\{(a,z^ka,\dots,z^{k(j-1)}a):a\in\B\}\quad\text{for
$0\leq k\leq j-1$.}
$$
Finally, the matrix $\{z^{kl}\}_{k,l=0}^{j-1}$ is invertible since
its determinant is a Van der Monde one. Invertibility of the latter
matrix implies that the union of $N_k$ for $0\leq k\leq j-1$ spans
$\B^j$. Hence $M=\B^j$ and therefore $u$ is a cyclic vector for $S$.
\end{proof}

For weighted bilateral shifts, the last lemma can be written in a
nicer form. Recall that if $|w_n|=|u_n|$ for any $n\in\Z$, then the
weighted bilateral shifts $T_{w,p}$ and $T_{u,p}$ are isometrically
similar for each $p\in[1,\infty]$. Indeed, consider the sequence
$\{d_n\}_{n\in\Z_+}$ defined as $d_0=1$,
$d_n=\widetilde{w}(1,n)/\widetilde{u}(1,n)$ for $n\geq1$ and
$d_n=\widetilde{u}(n+1,0)/\widetilde{w}(n+1,0)$ for $n<0$. Then
$|d_n|=1$ for each $n\in\Z_+$ and therefore the diagonal operator
$D$, which acts on the basic vectors according to the formula
$De_n=d_ne_n$ for $n\in\Z_+$, is an invertible isometry. One can
easily verify that $T_{w,p}=D^{-1}T_{u,p} D$. That is, $T_{w,p}$ and
$T_{u,p}$ are isometrically similar. In particular, any $T_{w,p}$ is
similar to $zT_{w,p}$ if $z\in\C$ and $|z|=1$. This observation
together with the above lemma lead to the following corollary.

\begin{corollary}\label{pow1} Let $T=T_{w,p}$ be a weighted bilateral
shift for which there exists $j\geq 2$ such that $T^j$ is cyclic.
Then $T\oplus T$ is cyclic.
\end{corollary}

\begin{proof} By Lemma~\ref{pow} the operator
$T\oplus zT\oplus{\dots}\oplus z^{j-1}T$ is cyclic, where $z=e^{2\pi
i/j}$. From the above observation it follows that $T$ is similar to
$z^kT$ for $0\leq k\leq j-1$. Hence the direct sum of $j$ copies of
$T$ is cyclic and therefore $T\oplus T$ is cyclic since $j\geq 2$.
\end{proof}

The next lemma provides a sufficient condition for a direct sum of
two weighted bilateral shifts to be non-cyclic.

\begin{lemma}\label{sumbws}Let $w$ be a bounded sequence of non-zero
complex numbers, $p_1,p_2\in[1,\infty]$ and
\begin{equation} \label{qq}
q=q(p_1,p_2)=\left\{\begin{array}{ll}\frac{p_1p_2}{p_1p_2-p_1-p_2}&\text{if
$p_1+p_2<p_1p_2$,}\\ \infty&\text{otherwise.}\end{array}\right.
\end{equation}
Assume also that there exists $m\in\Z_+$ such that
$a=\{a_n\}_{n\in\Z_+}\in\ell_q$, where
\begin{equation}\label{aa}
a_n=\frac{\widetilde{w}(m+1,m+n)}{\widetilde{w}(m-n+1,m)}\quad\text{for
$n\in\Z_+$.}
\end{equation}
Then $T_{w,p_1}\oplus T_{w,p_2}$ is non-cyclic.
\end{lemma}

\begin{proof} For shortness let $\B_p=\ell_p(\Z)$ if $1\leq
p<\infty$ and $\B_\infty=c_0(\Z)$. Consider the bilateral sequence
$\{d_n\}_{n\in\Z}$ defined by
$$
d_0=1,\ \ d_n=\prod_{j=1}^n\frac{w_j}{w_{2m+1-j}}\ \ \text{if $n>0$\
\ and}\ \ d_n=\prod_{j=1}^{|n|}\frac{w_{2m+j}}{w_{1-j}}\ \ \text{if
$n<0$.}
$$
It is straightforward to verify that
$d_{n+m}=d_{m-n}=(\widetilde{w}(m+1,2m))^{-1} \widetilde{w}(1,m)a_n$
for each $n>m$. Since $a\in\ell_q$, we have $d\in\ell_q(\Z)$. Let
$p_1'\in[1,\infty]$ be defined by the formula
$\frac1{p_1}+\frac1{p_1'}=1$. From the definition of $q$, we have
$\frac1{p_1'}\leq \frac1q+\frac1{p_2}$. Since $d\in\ell_q(\Z)$, we,
thanks to the H\"older inequality, have a bounded linear operator
$J:\B_{p_2}\to \B_{p_1'}$ defined on the canonical basis as
$Je_n=d_ne_{2m-n}$. It is straightforward to verify, by computing
the values of the operators on the basic vectors $(e_k,e_n)$, that
$(T_{w,p_1}\oplus S)(I\oplus J)=(I\oplus J)(T_{w,p_1}\oplus
T_{w,p_2})$, where $S$ is the bounded linear operator on $\B_{p_1'}$
defined as $Se_n=w_{n+1}e_{n+1}$ for $n\in\Z$.

Assume that $T_{w,p_1}\oplus T_{w,p_2}$ is cyclic. Since $I\oplus J$
has dense range, Lemma~\ref{quasi} implies that $T_{w,p_1}\oplus S$
is cyclic, which is impossible according to Lemma~\ref{adj}. Indeed,
if $1<p_1\leq\infty$, then $S=T^*_{w,p_1}$ and if $1\leq
p_1<\infty$, then $T_{w,p_1}=S^*$. Thus, in any case,
$T_{w,p_1}\oplus S$ is a direct sum of an operator with its dual.
\end{proof}

The following corollary is the particular case $p_1=p_2$ of the
above lemma.

\begin{corollary}\label{sumbws1}Let $w$ be a bounded sequence of non-zero
complex numbers, $p$ and $q=\infty$ if $p\leq 2$, $q=p/(p-2)$ if
$p>2$. Assume also that there exists $m\in\Z_+$ such that
$a=\{a_n\}_{n\in\Z_+}\in\ell_q$, where $a$ is defined in
$(\ref{aa})$. Then $T_{w,p}\oplus T_{w,p}$ is non-cyclic.
\end{corollary}

In order to prove the next proposition, we apply Lemma~\ref{sumbws}
in the case $\frac1{p_1}+\frac1{p_2}=1$.

\begin{proposition}\label{ttbws} Let $w=\{w_n\}_{n\in\Z}$ be a
bounded sequence of non-zero complex numbers. Then $T_{w,p}$ is
supercyclic if and only if $T_{w,p}\oplus T_{w,p'}$ is cyclic, where
$\frac1p+\frac1{p'}=1$.
\end{proposition}

\begin{proof} By Theorem~S, supercyclicity of $T_{w,p}$ does not depend on
$p$. In particular, $T_{w,p}$ is supercyclic if and only if
$T_{w,p'}$ is supercyclic. Thus, without loss of generality, we can
assume that $p\leq p'$. If $T_{w,p}$ is non-supercyclic, then by
Theorem~MS, $T_{w,p}$ satisfies the Supercyclicity Criterion. Since
an operator $T$ satisfies the Supercyclicity Criterion if and only
if $T\oplus T$ does, we have that $T_{w,p}\oplus T_{w,p}$ satisfies
the Supercyclicity Criterion and therefore $T_{w,p}\oplus T_{w,p}$
is cyclic. Since $\ell_p(\Z)\times \ell_p(\Z)$ is densely and
continuously embedded into $\ell_p(\Z)\times \ell_{p'}(\Z)$ if
$p'<\infty$ and into $\ell_p(\Z)\times c_0(\Z)$ if $p'=\infty$, we
see that cyclicity of $T_{w,p}\oplus T_{w,p}$ implies cyclicity of
$T_{w,p}\oplus T_{w,p'}$. Thus, $T_{w,p}\oplus T_{w,p'}$ is cyclic.

Assume now that $T_{w,p}$ is non-supercyclic. Theorem~S implies the
existence of $m\in\Z_+$ such that (\ref{sal4}) is not satisfied.
Then $a=\{a_n\}_{n\in\Z_+}\in\ell_\infty$, where $a$ is defined in
(\ref{aa}). It is easy to see from (\ref{qq}) that
$q=q(p,p')=\infty$. By Lemma~\ref{sumbws}, $T_{w,p}\oplus T_{w,p'}$
is non-cyclic.
\end{proof}

\begin{proposition}\label{pp1} Let $T$ be a bounded linear operator on a separable Banach
space $\B$. Then the condition {\rm (C1--C6)} of
Theorem~{\rm\ref{dich}} are related in the following way:
$$
{\rm (C4)}\,\Longleftarrow\,{\rm (C1)}\,\Longrightarrow\,{\rm
(C2)}\,\Longrightarrow\,{\rm (C3)}\,\Longrightarrow\,{\rm
(C6)}\,\Longrightarrow\,{\rm (C5)}.
$$
\end{proposition}

\begin{proof} By Theorem~SC, (C1) implies (C2). The implication ${\rm
(C1)}\,\Longrightarrow\,{\rm (C4)}$ follows from the same theorem
and the fact that $T$ satisfies the Supercyclicity Criterion if and
only if $T\oplus T$ does. Since the powers of a weakly supercyclic
operator are weakly supercyclic, we see that (C3) implies (C6). The
implications ${\rm (C2)}\,\Longrightarrow\,{\rm (C3)}$ and ${\rm
(C6)}\,\Longrightarrow\,{\rm (C5)}$ are trivial.
\end{proof}

\begin{proof}[Proof of Theorem~\ref{dich}] According to
Corollary~\ref{pow1} (C5) implies (C4). By Theorem~MS (C2) implies
(C1). Taking into account Proposition~\ref{pp1}, we see that it
suffices to show that (C4) implies (C2). If $T\oplus T$ is cyclic on
$\ell_p(\Z)\oplus \ell_p(\Z)$ then, since $p\leq 2$, it is cyclic on
$\ell_p(\Z)\oplus \ell_q(\Z)$. By Proposition~\ref{ttbws} $T$ is
supercyclic and therefore (C4) implies (C2). \end{proof}

\section{Proof of Theorem~\ref{suco}} \rm

Recall the following general result on universal families, see
\cite[p.~348--349]{ge1}. Let ${\cal F}=\{f_\alpha:\alpha\in A\}$ be
a family of continuous maps from a complete metric space $X$ to a
separable metric space $Y$. Then the set $\bigl\{x\in
X:\{f_a(x):a\in A\}\ \ \text{is dense in $Y$}\bigr\}$ of universal
elements for $\cal F$ is dense in $X$ if and only if the set
$\{(x,f_a(x)):x\in X,\ a\in A\}$ is dense in $X\times Y$.

A direct application of this result to the family $\{r(T):r\in\P\}$,
where $T$ is a bounded linear operator on a Banach space, gives us
the following theorem.

\begin{thmdc} \it Let $\B$ be a separable Banach space and $T:\B\to \B$ be a bounded
linear operator. Then the set of cyclic vectors for $T$ is dense in
$\B$ if and only if the set $\{(x,r(T)x):x\in \B,\ r\in \P\}$ is
dense in $\B\times \B$.\rm
\end{thmdc}

We say that a subset $A$ of a Banach space $\B$ is {\it cyclic} for
a bounded linear operator $T$ acting on $\B$ if
$\bigcup\limits_{r\in\P}r(T)(A)$ is dense in $\B$. Theorem~DC admits
the following refinement.

\begin{corollary} \label{dc} Let $\B$ be a separable Banach space,
$T:\B\to \B$ be a bounded linear operator and $A,B$ be two cyclic
subsets for $T$. Assume also that the point spectrum $\sigma_p(T^*)$
of the dual operator $T^*$ has empty interior. Then the set of
cyclic vectors for $T$ is dense in $\B$ if and only if for any $x\in
A$, $y\in B$ and $\epsilon>0$, there exist $u\in\B$ and $r\in \P$
such that $\|x-u\|<\epsilon$ and $\|y-r(T)u\|<\epsilon$.
\end{corollary}

\begin{proof} The 'only if' part follows immediately from
Theorem~DC. It remains to proof the 'if' part. Clearly
$$
B\subseteq M_\delta=\overline{\bigcup\limits_{r\in\P,\ x\in
A}r(T)(x+\delta U)}\quad \text{for any $\delta>0$},\ \text{where
$U=\{x\in \B:\|x\|<1\}$.}
$$
For any $\delta>0$ and $q\in\P$, we have $q(T)(M_\delta)\subseteq
M_\delta$ and therefore $q(T)(B)\subseteq M_\delta$. Since each
$M_\delta$ is closed and $B$ is cyclic for $T$, we have
$M_\delta=\B$ for any $\delta>0$. Let $\P^\dagger$ be the set of the
polynomials $q\in\P$, whose zeros are all in
$\C\setminus\sigma_p(T^*)$. Then $q(T)$ has dense range for any
$q\in\P^\dagger$. In particular, we see that the set
$$
q(T)\biggl(\ \bigcup\limits_{r\in\P,\ x\in A}r(T)(x+\delta
U)\biggr)=\bigcup\limits_{r\in\P,\ x\in A}r(T)(q(T)x+\delta q(T)(U))
$$
is dense in $\B$ for any $\delta>0$ and $q\in\P^\dagger$. Finally,
since $q(T)(U)\subseteq \|q(T)\|U$, we have
$$
\bigcup\limits_{r\in\P,\ x\in q(T)(A)}r(T)(x+\epsilon U)\quad
\text{is dense in $\B$ for any $\epsilon>0$ and $q\in\P^\dagger$.}
$$
Since $\sigma_p(T^*)$ has empty interior, using the definition of
$\P^\dagger$ and  cyclicity of $A$ for $T$, we obtain
$$
\bigcup_{q\in\P^\dagger} q(T)(A)\quad \text{is dense in $\B$}.
$$
The last two displays immediately imply that the set
$\{(x,r(T)x):x\in \B,\ r\in \P\}$ is dense in $\B\times\B$. By
Theorem~DC, the set of cyclic vectors for $T$ is dense in $\B$.
\end{proof}

We shall apply the above corollary to weighted bilateral shifts. The
following statement is a particular case of Corollary~\ref{dc}.

\begin{corollary}\label{dcs} Let $T$ be a bounded linear operator on
a Banach space $\B$ and $\{f_j\}_{j\in\Z}$ be a sequence of elements
of $\B$ such that $\spann\{f_j:j\in\Z\}$ is dense in $\B$ and
$Tf_j=f_{j+1}$ for each $n\in\Z$. Then the set of cyclic vectors for
$T$ is dense if and only if for any $n,k\in\N$, $n>k$ and any
$\epsilon>0$, there exists $r\in\P$ and $u\in\B$ such that
$\|f_{-k}-u\|\leq\epsilon$ and $\|f_{-n}-r(T)u\|\leq\epsilon$.
\end{corollary}

\begin{proof} The 'only if' part is a trivial consequence of
Theorem~DC. It remains to prove the 'if' part. Let $A=\{f_m:m<0\}$.
Since $Tf_j=f_{j+1}$ for each $n\in\Z$, we have
$$
\bigcup_{r\in\P}r(T)(A)=\spann\{f_j:j\in\Z\} \ \ \text{is dense in
$\B$.}
$$
Hence $A$ is a cyclic set for $T$. Let $x,y\in A$. Then $x=e_{-k}$
and $y=e_{-n}$ for some $k,n\in\N$. If $n\leq k$, then there is a
constant $c\in\C$ such that $r(T)x=y$, where $r(z)=cz^{k-n}$. In
particular $\|x-u\|=0<\epsilon$ and $\|y-r(T)u\|=0<\epsilon$ for
$u=x$ for any $\epsilon>0$. If $n>k$ and $\epsilon>0$, then by the
assumptions, there exists $r\in\P$ and $u\in\B$ such that
$\|x-u\|<\epsilon$ and $\|y-r(T)u\|<\epsilon$. It remains to apply
Corollary~\ref{dc}.
\end{proof}

\begin{proposition} \label{suco3} Let $T$ be a bounded linear operator on
a Banach space $\B$ and $\{f_j\}_{j\in\Z}$ be a sequence of elements
of $\B$ such that $\spann\{f_j:j\in\Z\}$ is dense in $\B$ and
$Tf_j=f_{j+1}$ for each $j\in\Z$. Assume also that that
\begin{equation}\label{u123}
\inf\{\|f_{-m}\|\|f_{(m-a)j}\|^{1/j}:j\in\N\ m\geq a\}=0\quad
\text{for any $a\in\N$}.
\end{equation}
Then $T$ has dense set of cyclic vectors. \end{proposition}

\begin{proof} Let $\epsilon>0$ and $n,k\in\N$ are such that $n>k$.
For any $j\in\N$ and $m\geq n>k$ we consider $x_m\in\B$ and a
polynomial $q_{j,m}$ defined by
$$
x_m=f_{-k}-\frac{\epsilon}{\|f_{-m}\|}f_{-m}\quad \text{and}\quad
q_{j,m}(z)=-\frac{\|f_{-m}\|z^{m-n}}{\epsilon}\sum_{l=0}^{j-1}
\biggl(\frac{\|f_{-m}\|z^{m-k}}{\epsilon}\biggr)^l.
$$
Using the fact that $Tf_l=f_{l+1}$ for each $l\in\Z$ and the
summation formula for a finite geometric progression, one can easily
verify that
$q_{j,m}(T)x_m=f_{-n}-(\|f_{-m}\|/\epsilon)^jf_{(m-k)j-n}$. Hence
$$
\|q_{j,m}(T)x_m-f_{-n}\|=(\|f_{-m}\|/\epsilon)^j\|f_{(m-k)j-n}\|.
$$
Now let $a=n+k$ and assume that $m\geq a$. Then
$$
\|f_{(m-k)j-n}\|=\|f_{(m-a)j+n(j-1)}\|=\|T^{n(j-1)}f_{(m-a)j}\|\leq
\|T^n\|^{j-1}\|f_{(m-a)j}\|.
$$
From the last two displays, we obtain
$$
\|q_{j,m}(T)x_m-f_{-n}\|\leq(\|f_{-m}\|/\epsilon)^j\|T^n\|^{j-1}\|f_{(m-a)j}\|.
$$
Now, from (\ref{u123}) it follows that $j\in\N$ and $m\geq a$ can be
chosen in such a way that the right hand side of the above
inequality does not exceed $\epsilon$. In this case
$\|q_{j,m}(T)x_m-f_{-n}\|\leq\epsilon$. Since from the definition of
$x_m$ it follows that $\|x_m-f_{-k}\|=\epsilon$, Corollary~\ref{dcs}
implies that $T$ has dense set of cyclic vectors.
\end{proof}

\begin{proof}[Proof of Theorem~\ref{suco}]
As we have already mentioned, if $|w_n|=|u_n|$ for any $n\in\Z$,
then the weighted bilateral shifts $T_{w,p}$ and $T_{u,p}$ are
isometrically similar for each $p\in[1,\infty]$. Thus, we can,
without loss of generality, assume that $w_n>0$ for each $n\in\Z$.

Let $f_n=c_ne_{-n}$ for $n\in\Z$, where $c_n=1$ if $n=0$,
$c_n=\widetilde w(1-n,0)$ if $n>0$ and $c_n=(\widetilde
w(1,-n))^{-1}$ if $n<0$. It is straightforward to see that
$Tf_n=f_{n+1}$ for each $n\in\Z$. Clearly
$\spann\{f_n:n\in\Z\}=\spann\{e_n:n\in\Z\}$ is dense. Now, since
$\|f_n\|=c_n$, (\ref{u123}) is equivalent to (\ref{a123}). It
remains to apply Proposition~\ref{suco3}.
\end{proof}

\section{Concluding remarks \label{s6}} \rm

It is worth mentioning another dichotomy for weighted bilateral
shifts, provided in \cite{shk2}. Namely if $T$ is a weighted
bilateral shift on $\ell_p(\Z)$ with $1<p<\infty$ then either $T$ is
supercyclic or $T^nx/\|T^nx\|$ is weakly convergent to zero as
$n\to\infty$ for each non-zero $x\in\ell_p(\Z)$. The same holds true
for weighted bilateral shifts on $c_0(\Z)$ and fails for weighted
bilateral shifts on $\ell_1(\Z)$. We would like to stress that the
following problem remains open.

\begin{problem}\label{p1} Characterize cyclic weighted bilateral
shifts on $\ell_p(\Z)$ for $1\leq p<\infty$ and on $c_0(\Z)$.
\end{problem}

As mentioned in \cite{aba}, it is not known whether certain specific
weighted bilateral shifts are cyclic. For instance, let
$0<\alpha\leq 1$, $w_n=1$ if $n\leq 1$ and $w_n=1-n^{-\alpha}$ if
$n\geq 2$. It is easy to  see that the point spectrum of $T^*_{w,1}$
is non-empty (it coincides with the unit circle) and therefore,
according to Herrero, $T_{w,1}$ is non-cyclic for any $\alpha$. From
Theorem~A\!$^2$\nobreak\hskip-1pt\nobreak G it follows that
$T_{w,p}$ is cyclic for $1<p\leq \infty$ when $\alpha>1/2$.

\begin{problem}\label{p2} Let $w$ is the above weight sequence with
$\alpha=1/2$. For which $p\in(1,\infty]$ is $T_{w,p}$ cyclic?
\end{problem}

We conjecture that the answer to the following question must be
affirmative.

\begin{problem}\label{p3} Let $2<p\leq\infty$. Does there exist a
weighted bilateral shift $T=T_{w,p}$ such that $T^n$ is cyclic for
any $n\in\N$ and $T$ is not weakly supercyclic?
\end{problem}

Note that Salas \cite{sal} proved that $I+T$ is hypercyclic for any
(unilateral) backward weighted shift $T$. We would like to raise the
following problem.

\begin{problem}\label{p4} Characterize weighted bilateral shifts $T_{w,p}$
for which $I+T$ is hypercyclic. What about supercyclicity?
\end{problem}

What the author has been able to observe so far is that for $1\leq
p\leq 2$, supercyclicity of $I+T$ implies supercyclicity of $T$ for
any weighted bilateral shift $T$ on $\ell_p(\Z)$.

Finally, we would like to raise the following problem.

\begin{problem}\label{p5} Does there exists a bounded linear
operator $T$ on a separable Banach space such that $T\oplus T$ is
cyclic and $T^2$ is non-cyclic?
\end{problem}

\bigskip

{\bf Acknowledgements.} The author would like to thank the referee
for numerous helpful remarks and suggestions.

\vfill\break

\small\rm

\vskip1truecm

\scshape

\noindent Stanislav Shkarin

\noindent Queen's University Belfast

\noindent Department of Pure Mathematics

\noindent University road, Belfast BT7 1NN, UK

\noindent E-mail address: \qquad {\tt s.shkarin@qub.ac.uk}

\end{document}